\documentclass[a4paper]{amsart}
\usepackage{tikz}
\usetikzlibrary{patterns}
\usetikzlibrary{scopes}
\usepackage{hyperref}

\newtheorem{theorem}{Theorem}[section]
\newtheorem{proposition}[theorem]{Proposition}
\newtheorem{corollary}[theorem]{Corollary}

\theoremstyle{definition}
\newtheorem{definition}[theorem]{Definition}

\title
    [A multiply connected Rudin-Carleson theorem]
    {A Rudin-Carleson theorem for multiply connected domains with interpolation}
\author{Benedikt Steinar Magnusson}
\email{bsm@hi.is}
\author{Bergur Snorrason}
\email{bergur@hi.is}

\begin{document}
\begin{abstract}
Using an annular version of the F. and M. Riesz theorem,
    we prove a generalization of the Rudin-Carleson theorem for finitely connected bounded domains.
That is,
    for a continuous function on a closed set in the boundary of measure zero there is a holomorphic function on the
    domain continuous to the boundary.
Furthermore, this can be done with interpolation at finitely many points in the domain.
The proof relies on an annular version of the F. and M. Riesz theorem.
\end{abstract}
\maketitle
\section{Introduction}
\label{sec:01}
\noindent
The Rudin-Carleson theorem \cite{Car:1957, Rud:1956} states that a continuous function on a closed subset
    of the unit circle $\mathbb{T} = \{z \in \mathbb{C} \,;\, |z| = 1\}$
    of arclength measure $0$ can be extended to a continuous function on the closed unit disc
    $\overline{\mathbb{D}} = \{z \in \mathbb{C} \,;\, |z| \leq 1\}$
    that is holomorphic on the open unit disc
    $\mathbb{D} = \{z \in \mathbb{C} \,;\, |z| < 1\}$.
This result has been generalized and refined by various authors \cite{Dan:2010, Gam:1969, Bru:2022, Glo:1980}.
These generalizations either consider the unit disc as the domain of definition, as in the original result,
    or look at the problem from the viewpoint of uniform algebras on a compact Hausdorff space.
The goal of this paper is proving the following generalization of the theorem
    for multiply connected domains in $\mathbb{C}$.
\begin{theorem}
\label{thm:1.1}
Let
    $\Omega \subset \mathbb{C}$ be a $k$-connected domain,
    with $k > 1$,
    such that the boundary of $\Omega$ consists of $k$ pairwise disjoint Jordan curves,
    $E \subset \partial \Omega$ be a closed arclength null set,
    $f \colon E \rightarrow \mathbb{C}$ a continuous function, and
    $M \colon \partial \Omega \rightarrow ]0, \infty[$ a continuous function such that $|f| < M$ holds on $E$.
Then there exists a continuous function $F$ on $\overline{\Omega}$ that is holomorphic on $\Omega$, $F|_E = f$, and
    $|F| < M$ on $\partial \Omega$.
\end{theorem}
The notion of an arclength null subset of $\partial \Omega$ is defined in Definition~\ref{def:2.7}.
To prove this theorem we follow Bishop \cite{Bis:1962},
    who gave a description of how smalls sets can be to ensure that continuous functions can be extended.
It is not immediately clear that the classical Rudin-Carleson theorem is a corollary of
    Bishop's theorem (Thm.~\ref{thm:2.2}),
    but it follows from the F. and M. Riesz theorem (Thm.~\ref{thm:2.3}).
Theorem~\ref{thm:1.2} is an annular F. and M. Riesz theorem.
\begin{theorem}
\label{thm:1.2}
Let
    $r_0 \in ]0, 1[$,
    $A = \{z \in \mathbb{C} \,; r_0 < |z| < 1\}$,
    $\mu$ be a complex measure on $\partial A = \mathbb{T} \cup (r_0 \mathbb{T})$, and
    $\widehat{\mu}^0_j$ and $\widehat{\mu}^1_j$, for $j \in \mathbb{Z}$, be the Fourier coefficients of $\mu$
        when restricted to $r_0 \mathbb{T}$ and $\mathbb{T}$ respectively, given by
\begin{equation*}
    \widehat{\mu}^0_j
    =
    \int_{r_0 \mathbb{T}}
        z^{-j}\ d\mu
    \quad \text{ and } \quad
    \widehat{\mu}^1_j
    =
    \int_{\mathbb{T}}
        z^{-j}\ d\mu,
    \quad
    j \in \mathbb{Z}.
\end{equation*}
If $\widehat{\mu}^0_j = -\widehat{\mu}^1_j$ for all $j \in \mathbb{Z}$,
    then every subset of $\partial A$ of arclength measure $0$ is a null set with respect to $\mu$.
\end{theorem}

In Section \ref{sec:02}
    we discuss how the Rudin-Carleson theorem follows from Bishop's theorem,
    as well as
    showing how Theorem~\ref{thm:1.1} follows from Theorem~\ref{thm:1.2},

In Section \ref{sec:03} we prove Theorem \ref{thm:1.2}.

In Section \ref{sec:04} we prove a version of Theorem \ref{thm:1.1} that additionally allows interpolation at finite
    points in $\Omega$.
The upper bound $M$ can not be maintained in the theorem as it would possibly break the maximum principle.
Consider, for example, $\Omega = \mathbb{D}$ and $M = 1$.
If we require our extension to satisfy $F(0) = 1$ then,
    by the maximum modulus principle,
    no extension exists.
See also Izzo \cite{Izz:2018}.

\section{A Rudin-Carleson theorem for multiply connected domains}
\label{sec:02}
\noindent
In the 50's Carleson \cite{Car:1957} and Rudin \cite{Rud:1956} proved the following, independently.
\begin{theorem}[Rudin-Carleson theorem]
\label{thm:2.1}
Let $E \subset \mathbb{T}$ be closed and of arclength measure $0$ and $f \colon E \rightarrow \mathbb{C}$ be continuous.
There then exists a continuous function $F$ on $\overline{\mathbb{D}}$ that is holomorphic on $\mathbb{D}$ and
    $F|_E = f$.
\end{theorem}
Bishop \cite{Bis:1962} later generalized this greatly.
\begin{theorem}[Bishop's theorem]
\label{thm:2.2}
Let 
    $X$ be a compact Hausdorff space,
    $B$ be a closed subset of the continuous functions on $X$ with the supremum norm, denoted by $\mathcal{C}(X)$,
    $E \subset X$ be a closed null set with respect to all measures $\mu$ such that $\int_X g\ d\mu = 0$ for all
        $g \in B$,
    $f \colon E \rightarrow \mathbb{C}$ be continuous, and
    $M \colon X \rightarrow ]0, \infty[$ be continuous and such that $|f| < M$ on $E$.
Then there exists $F \in B$ such that $|F| < M$ and $F|_E = f$.
\end{theorem}

To prove the Rudin-Carleson theorem using Bishop's theorem we set $X = \mathbb{T}$ and $B$
    as the restriction of the continuous function on $\overline{\mathbb{D}}$, that are holomorphic on $\mathbb{D}$,
    to $\mathbb{T}$.
We just need to show that if $E$ is of arclength measure $0$ then it is a null set with respect to all measures $\mu$
    such that $\int_X g\ d\mu = 0$, for all $g \in B$.
This is a consequence of the F. and M. Riesz theorem.
\begin{theorem}[F. and M. Riesz theorem]
\label{thm:2.3}
Let $\mu$ be a complex measure on $\mathbb{T}$ and define its Fourier coefficients by
\begin{equation*}
    \widehat{\mu}_j
    =
    \int_{\mathbb{T}} z^{-j}\ d\mu,
    \quad
    j \in \mathbb{Z}.
\end{equation*}
Then
\begin{enumerate}
\item[(i)]
If $\widehat{\mu}_{-j} = 0$, for all $j \in \mathbb{N}^*$, then every $E \subset \mathbb{T}$ of arclength measure $0$
    is a null set with respect to $\mu$.
Additionally, we have, by the Radon-Nikodym theorem, that there exists $F \in L^1(\mathbb{T})$,
    such that $d\mu = F\,d\sigma$, where $\sigma$ is the arclength measure standardized as a probability
    measure,
    and $F$ are the non-tangential boundary values of $f(z) = \sum_{j = 0}^\infty \widehat{\mu}_j z^j$, $|z| < 1$.

\item[(ii)]
If $\widehat{\mu}_j = 0$, for all $j \in \mathbb{N}^*$, then $d\mu = F\, d\sigma$
    where $F$ are the non-tangential boundary values of $f(z) = \sum_{j = 0}^\infty \widehat{\mu}_j z^{-j}$, $|z| > 1$.
\end{enumerate}
\end{theorem}

Classically, (ii) is not included.
It follows from (i) by considering the conjugate of the measure $\mu$.

The F. and M. Riesz theorem can be found as Theorem $17.13$ in Rudin \cite{Rudin:1987}.
The proof of the Rudin-Carleson theorem using Bishop's theorem is then concluded by noting that the functions
    $z \mapsto z^j$, for $j \in \mathbb{N}^*$, all belong to $B$.

\medbreak
We now turn to proving Theorem \ref{thm:1.1} in a similar manner using Theorem \ref{thm:1.2}, which we will prove in
    Section \ref{sec:03}.
We will start by proving Theorem~\ref{thm:1.1} for the case when $\Omega$ is annular.

If $A = \{z \in \mathbb{C} \,;\, r_0 < |z| < 1\}$,
    for some $r_0 > 0$,
    then $\partial A = \mathbb{T} \cup r_0 \mathbb{T}$.
We say that $E \subset \partial A$ is of arclength measure $0$,
    if both $E \cap \mathbb{T}$ and $E \cap r_0 \mathbb{T}$, have arclength measure $0$.

\begin{proposition}
\label{prop:2.4}
Let
    $r_0 \in ]0, 1[$,
    $A = \{z \in \mathbb{C} \,; r_0 < |z| < 1\}$,
    $E \subset \partial A$ be closed and of arclength measure $0$,
    $f \colon E \rightarrow \mathbb{C}$ be a continuous function, and
    $M \colon \partial A \rightarrow ]0, \infty]$ be a continuous function such that $|f| < M$ on $E$.
Then there exists a continuous function $F$ on $\overline{A}$ that is holomorphic on $A$, $|F| < M$ and $F|_E = f$.
\end{proposition}
\begin{proof}
\looseness = -1
Let $\mu$ be a complex measure on $\partial A$ such that $\int_{\partial A} g\ d\mu = 0$
    for all function $g$ on $\partial A$ that extend continuously to holomorphic functions on $A$.
To show that $E$ satisfies the conditions of Bishop's theorem we need to show that it is a null set with respect to
    $\mu$.
This will conclude the proof.
By assumption
\begin{equation*}
    0
    =
    \int_{\partial A} z^j\ d\mu
    =
    \int_{r_0 \mathbb{T}} z^j\ d\mu
    +
    \int_{\mathbb{T}} z^j\ d\mu,
    \quad
    j \in \mathbb{Z},
\end{equation*}
so $\mu$ satisfies the conditions of Theorem \ref{thm:1.2}.
So $E$ is a null set with respect to $\mu$.
\end{proof}

The specific case when $f = 0$ is not interesting, since we can chose $F = 0$.
The ability to chose a non-constant extension will be useful later.
\begin{corollary}
The extension $F$ in the previous proposition can be chose such that it is non-constant.
\end{corollary}
\begin{proof}
Assume $f$ is constant, $p \in \partial A \setminus E$, and $w \in \mathbb{C}$ such that $f \neq w$ and $|w| < M(p)$.
We can then apply the previous Proposition with $E \cup \{p\}$ and $f$ extended to $p$ with $f(p) = w$.
\end{proof}

To adapt this result to multiply connect domains we need a variant of the Riemann mapping theorem.
Recall that a domain in the Riemann sphere is \emph{$k$-connect} if its complement has $k$ connected components,
    and it is \emph{multiply connected} if it is $k$-connected for some $k > 1$.
If $k = 2$ we say the domain is \emph{doubly connected}.

\begin{theorem}[Doubly connected Riemann mapping theorem]
\label{thm:2.6}
Let $\Omega$ be a bounded doubly connected domain in $\mathbb{C}$ whose boundary consists of two disjoint Jordan curves.
Then there exists a biholomorphic mapping from $\Omega$ to an annulus $A = \{z \in \mathbb{C} \,; r_0 < |z| < 1\}$,
    for some $r_0 > 0$, that extends continuously to the boundary.
\end{theorem}

This is a consequence of Theorem $4.2.1$ in Krantz \cite{Kra:2006} along with an application of the
    Caratheodory extension theorem, which is Theorem $13.2.3$ in Greene and Krantz \cite{Krantz:OCV}.

\medbreak
Let $\Omega$ be a multiply connected domain in $\mathbb{C}$ whose boundary is composed of pairwise disjoint
    Jordan curves $\Gamma_1, \dots, \Gamma_k$.
By the Jordan curve theorem each curve $\Gamma_j$, for $j = 1, \dots, k$, section $\mathbb{C}$ into two disjoint sets,
    one bounded and the other unbounded.
By the classic Riemann mapping theorem there exists a biholomorphic function $\Phi$ between the bounded component and
    $\mathbb{D}$, and by the Caratheodory extension theorem it extends continuously to the boundary.
\begin{definition}
\label{def:2.7}
Let $\Omega \subset \mathbb{C}$ be a multiply connected domain whose boundary is composed of pairwise disjoint
    Jordan curves $\Gamma_1, \dots, \Gamma_k$.
We say that $E \subset \partial \Omega$ is an \emph{arclength null set} if $\Phi_j(E \cap \Gamma_j)$
    is of arclength measure $0$, where $\Phi_j$ is a biholomorphic function from the bounded domain with boundary
    $\Gamma_j$ and $\mathbb{D}$.
\end{definition}
Note that the biholomorphic maps $\Phi_j$ are not unique, so this definition must be justified.
To do this we use a result of Fatou, which is Theorem $13.4.11$ in \cite{Krantz:OCV}.
It says that if $f$ is a continuous function on $\overline{\mathbb{D}}$ that is holomorphic on $\mathbb{D}$ and $f$
    vanishes on a subset of $\mathbb{T}$ of arclength measure greater than $0$ than $f = 0$.

Let $\Phi_1$ and $\Phi_2$ be biholomorphic functions from the bounded domain with boundary $\Gamma_j$, for some
    $j \in \{1, \dots, k\}$,
    that are continuous to the boundary and such that $F = \Phi_1(E)$ is of arclength measure $0$.
By the Rudin-Carleson theorem there exists a non-constant function $f$ that is continuous on $\overline{\mathbb{D}}$,
    holomorphic on $\mathbb{D}$, and
    $f|_F = 0$.
Define $g = f \circ \Phi_1 \circ \Phi_2^{-1}$ and note that $g$ is non-constant and $g|_{\Phi_2(E)} = 0$.
So $\Phi_2(E)$ is also of arclength measure $0$.

We can now use Proposition~\ref{prop:2.4} to prove Theorem \ref{thm:1.1}.
Note that we only prove it for connected sets.
The general case breaks down to the connected case by considering each connected component on its own.
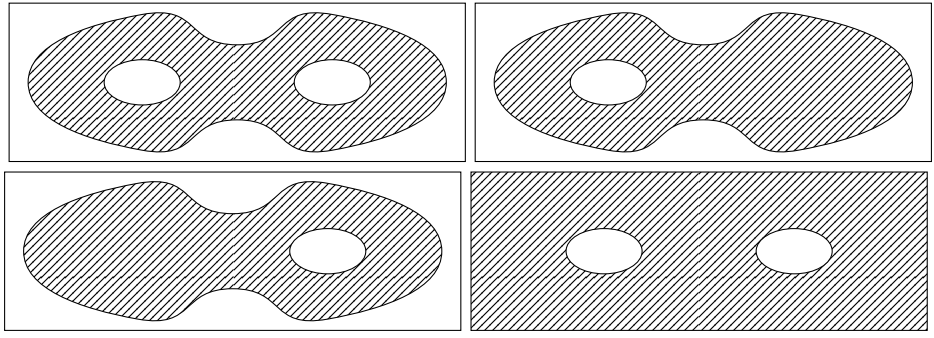
\begin{figure}
    { 
        \def\outer
        {
            plot [name = A, smooth cycle, tension = 1.0] coordinates
            {
                (0, 3.317)
                (5, 6)
                (11, 0)
                (5, -6)
                (0, -3.317)
                (-5, -6)
                (-11, 0)
                (-5, 6)
            }
        }
        \def\innerone{(-5, 0) circle (2)}
        \def\innertwo{(5, 0) circle (2)}
        \def\border{(-12, -7) rectangle (12, 7)}
        \begin{tikzpicture}[x = 0.5cm, y = 0.3cm, scale = 0.5]
            \clip[draw] \border;
            \draw \outer;
            \draw \innerone;
            \draw \innertwo;
            \fill[even odd rule, pattern = north east lines] \outer \innerone \innertwo;
        \end{tikzpicture}
        \begin{tikzpicture}[x = 0.5cm, y = 0.3cm, scale = 0.5]
            \clip[draw] \border;
            \draw \outer;
            \draw \innerone;
            \fill[even odd rule, pattern = north east lines] \outer \innerone;
        \end{tikzpicture}

        \vspace{2.5pt}
        \begin{tikzpicture}[x = 0.5cm, y = 0.3cm, scale = 0.5]
            \clip[draw] \border;
            \draw \outer;
            \draw \innertwo;
            \fill[even odd rule, pattern = north east lines] \outer \innertwo;
        \end{tikzpicture}
        \begin{tikzpicture}[x = 0.5cm, y = 0.3cm, scale = 0.5]
            \clip[draw] \border;
            \draw \innerone;
            \draw \innertwo;
            \fill[even odd rule, pattern = north east lines] \innerone \innertwo \border;
        \end{tikzpicture}
    }

    \caption
    {
        Illustrated is an example of $\Omega$ from the proof of Theorem~\ref{thm:1.1},
            were $k = 3$.
        The figure in the upper left shows $\Omega$, while
            the other three show $D_{1, 2}$, $D_{1, 3}$, and $D_{2, 3}$,
            depending on how the components of the boundary are numbered.
    }
\end{figure}

\begin{proof}[Proof of Theorem~\ref{thm:1.1}]
Let
    $\Gamma_1, \dots, \Gamma_k$ be the connected components of $\partial \Omega$,
    $D_{j, \ell}$, for $j \neq \ell$,
        be the doubly connected domain in the Riemann sphere $\mathbb{C}_\infty$ that has boundary
        $\Gamma_j \cup \Gamma_\ell$ and $\Omega \subset D_{j, \ell}$,
    $D_j$ is the simply connected domain in the Riemann sphere that has boundary $\Gamma_j$ and $\Omega \subset D_j$,
    and $\gamma > 0$ be small enough that $|f| + 2\gamma < M$, on $E$.
By the Rudin-Carleson theorem there exist continuous functions $F_j \colon \overline{D}_j \rightarrow \mathbb{C}$
    that are holomorphic on $D_j$,
    $F_j = f$ on $E \cap \Gamma_j$, and 
    $|F_j| < M - 2\gamma$ on $\Gamma_j$.
We then choose $\varepsilon > 0$ small enough that
\begin{equation}
\label{eq:1}
    (1 + \varepsilon)^{k - 1} |F_j(z)|
    <
    |F_j(z)| + \gamma
    < M(z) - \gamma,
    \quad
    z \in \Gamma_j,\ 
    j = 1, \dots, k
\end{equation}
and $\delta > 0$ small enough that 
\begin{equation}
\label{eq:2}
    \delta (1 + \varepsilon)^{k - 2}
    \sum_{\substack{\ell = 1 \\ \ell \neq j}}^k
        |F_\ell(z)|
    <
    \gamma,
    \quad
    z \in \Gamma_j,\ 
    j = 1, \dots, k.
\end{equation}
For $j \neq \ell$,
    by Proposition~\ref{prop:2.4} and Theorem~\ref{thm:2.6},
    there exist continuous functions $h_{j, \ell}$ on
    $\overline{D_{j, \ell}}$ that are holomorphic on $D_{j, \ell}$ and satisfy,
\begin{enumerate}
    \item[(A1)] $h_{j, \ell} = 1$ on $E \cap \Gamma_j$,
    \item[(A2)] $h_{j, \ell} = 0$ on $E \cap \Gamma_\ell$,
    \item[(A3)] $|h_{j, \ell}| < 1 + \varepsilon$ on $\Gamma_j$, and
    \item[(A4)] $|h_{j, \ell}| < \min\{\delta, 1 + \varepsilon\}$ on $\Gamma_\ell$.
\end{enumerate}
By points (A3) and (A4), along with the maximum modulus principle, we have that $|h_{j, \ell}| < 1 + \varepsilon$ on
    $D_{j, \ell}$.
We can now define continuous function $h_j$ on $\overline{\Omega}$ by
\begin{equation*}
    h_j
    =
    \prod_{\substack{\ell = 1 \\ \ell \neq j}}^k
        h_{j, \ell}
\end{equation*}
satisfying
\begin{enumerate}
    \item[(B1)] $h_j = 1$ on $E \cap \Gamma_j$,
    \item[(B2)] $h_j = 0$ on $E \setminus \Gamma_j$,
    \item[(B3)] $|h_j| < (1 + \varepsilon)^{k - 1}$ on $\Gamma_j$, and
    \item[(B4)] $|h_j| < \delta(1 + \varepsilon)^{k - 2}$ on $\partial \Omega \setminus \Gamma_j$.
\end{enumerate}
Point (B1) follows from (A1) and (B2) follows from (A2).

We are now ready to piece together our extension, letting
\begin{equation*}
    F
    =
    \sum_{\ell = 1}^k F_\ell \cdot h_\ell.
\end{equation*}
Note first that $F$ is continuous on $\overline{\Omega}$ and holomorphic on $\Omega$
    and if $w \in E$ there is an $j$ such that $w \in \Gamma_j$ and
\begin{equation*}
    F(w)
    =
    \sum_{\ell = 1}^k F_\ell(w) h_\ell(w)
    =
    F_j(w)
    =
    f(w),
\end{equation*}
by point (B1) and (B2).
So $F|_E = f$.
Additionally, we have for $w \in \partial \Omega$ an $j$ such that $w \in \Gamma_j$ and
\begin{equation*}
    |F(w)|
    \leq
    \sum_{\ell = 1}^k |F_\ell(w)h_\ell(w)|
    \leq
    (1 + \varepsilon)^{k - 1} |F_j(w)|
    +
    \delta(1 + \varepsilon)^{k - 2}
    \sum_{\substack{\ell = 1 \\ \ell \neq j}}^k
        |F_\ell(w)|,
\end{equation*}
by (B3) and (B4).
By \eqref{eq:1} and \eqref{eq:2} we have that
\begin{equation*}
    |F(w)|
    \leq
    M(w) - \gamma + \gamma
    =
    M(w).
    \qedhere
\end{equation*}
\end{proof}

\section{Annular F. and M. Riesz theorem}
\label{sec:03}
In this section we will prove Theorem~\ref{thm:1.2} which, as shown in Section \ref{sec:02},
    proves Theorem~\ref{thm:1.1}.
The method of the proof involves decomposing the measure in a way similar to how a holomorphic function on an annulus
    can be decomposed into the sum of holomorphic functions on discs whose intersection is the annulus.

\medbreak
We will start by recalling some basic concepts of the decomposition of complex measures into their Fourier series.
Let us fix $a > 0$ and a complex measure $\mu$ on $a \mathbb{T}$ and define the $j$-th Fourier coefficient of $\mu$ by
\begin{equation*}
    \widehat{\mu}_j
    =
    \int_{a\mathbb{T}} z^{-j}\ d\mu(z),
    \quad
    j \in \mathbb{Z}.
\end{equation*}
We get a natural upper bound for these coefficients by
\begin{equation}
\label{eq:3}
    |\widehat{\mu}_j|
    \leq
    \int_{a\mathbb{T}} |z^{-j}|\ d|\mu|
    =
    a^{-j} |\mu|(a\mathbb{T}),
    \quad
    j \in \mathbb{Z}.
\end{equation}
So we have a linear operator $\alpha$ from the Banach space of complex measures to the sequence space
    $\mathbb{C}^{\mathbb{Z}}$, given by
    $\alpha(\mu) = (\dots, \widehat{\mu}_{-1}, \widehat{\mu}_0, \widehat{\mu}_1, \dots)$.
To see that $\alpha$ is injective we set $\mu$ such that $\alpha(\mu) = 0$.
Then $\int_{\mathbb{T}} z^j\ d\mu = 0$ and if $f$ is continuous on $\mathbb{T}$ then
\begin{equation*}
    \int_{\mathbb{T}} f\ d\mu
    =
    \lim_{j \rightarrow \infty}
    \int_{\mathbb{T}} f_j\ d\mu
    =
    0
\end{equation*}
where $f_j$ are a linear combination of $\{z \mapsto z^j \,; j \in \mathbb{Z}\}$ and a uniform approximation of $f$.
So $\mu = 0$, and $\alpha$ is injective.
We do not, however, have that $\alpha$ is surjective.
The proof of Theorem~\ref{thm:1.2} effectively boils down to justifying that a given sequences in
    $\mathbb{C}^{\mathbb{Z}}$ is in the image of $\alpha$.
\begin{proof}[Proof of Theorem~\ref{thm:1.2}]
Let
    $E \subset \partial A$ be of arclength measure $0$,
    $\widehat{\mu}_j = \widehat{\mu}^1_j$,
    $\mu^0$ and $\mu^1$ be the restrictions of $\mu$ to $r_0 \mathbb{T}$ and $\mathbb{T}$, respectively, and
\begin{equation*}
    f(z) 
    =
    \sum_{j = 1}^\infty \widehat{\mu}_j z^j,
    \quad
    z \in \mathbb{D}.
\end{equation*}
The convergence of $f$ follows from \eqref{eq:3}, with $a = 1$.
We do not know if we can use $f$ to define a measure on $\mathbb{T}$, but we can use it to define a measure on
    $r_0 \mathbb{T}$ by $d\lambda^0 = f\,d\sigma$, satisfying
\begin{equation*}
    \widehat{\lambda}^0_k
    =
    \int_{r_0 \mathbb{T}}
        z^{-k}\ d\lambda^0(z)
    =
    \int_{r_0 \mathbb{T}}
        f(z)z^{-k}\ d\sigma(z)
    =
    \left \{
    \begin{array}{l l}
        \widehat{\mu}_k, & k \in \mathbb{N}^*, \\
        0, & k \not \in \mathbb{N}^*,
    \end{array}
    \right .
\end{equation*}
where $\sigma$ denotes the standardized arclength measure.
We define $\eta^0 = \mu^0 - \lambda^0$.
    By Theorem~\ref{thm:2.3}(ii)
    we have that $E$ is a null set with respect to $\eta^0$ and by construction it is a
    null set with respect to $\lambda^0$, and consequently also with respect to $\mu^0$.
By \eqref{eq:3}, with $a = r_0$, we have that
    $g(z) = \sum_{j = 1}^\infty \widehat{\mu}_{-j} z^{-j}$ converges when $|z| > r_0$.
So we define a measure on $\mathbb{T}$ by $d\eta^1 = g\,d\sigma$, satisfying
\begin{equation*}
    \widehat{\eta}^1_k
    =
    \int_{\mathbb{T}}
        z^{-k}\ d\eta^1(z)
    =
    \int_{\mathbb{T}}
        g(z)z^{-k}\ d\sigma(z)
    =
    \left \{
    \begin{array}{l l}
        0, & k \in \mathbb{N}, \\
        \widehat{\mu}_k, & k \not \in \mathbb{N},
    \end{array}
    \right .
\end{equation*}
and $\lambda^1 = \mu^1 - \eta^1$.
Again, we have that $E$ is a null set with respect to $\eta^1$ and $\lambda^1$,
    and hence,
    also with respect to $\mu^1$.
Since $E$ is a null set with respect to $\mu$ if and only if it is a null set with respect to
    $\mu_1$ and $\mu_2$, the proof is concluded.
\end{proof}

\section{A Rudin-Carleson theorem with interpolation}
\label{sec:04}
Proposition~\ref{prop:2.4} does not cover the degenerate annulus $A = \mathbb{D} \setminus \{0\}$,
    that is the punctured disc.
Due to the maximum modulus principle we can not expected to maintain the upper bounded, even when the extension exists.
The goal of this section is to show the existence of an extension, without an upper bound.
\begin{theorem}
Let
    $\Omega \subset \mathbb{C}$ be an open $(k + \ell)$-connected domain,
    with $k > 1$,
    such that the boundary of $\Omega$ consists of $k$ pairwise disjoint Jordan curves and $\ell$ points,
    $E \subset \partial \Omega$ be a closed arclength null set, and
    $f \colon E \rightarrow \mathbb{C}$ a continuous function.
Then there exists a continuous function $F$ on $\overline{\Omega}$ that is holomorphic on $\Omega$ and $F|_E = f$.
\end{theorem}
\looseness = -1
When we say that $E$ is an arclength null set we mean it in the same sense as Definition~\ref{def:2.7},
    but we ignore the connected components of $\partial \Omega$ that are points.
That is, we say that $E$ is an \emph{arclength null set} if $E \cap (\Gamma_1 \cup \dots \cup \Gamma_k)$
    is an arclength null set.
\begin{proof}
Let $\Gamma_1, \dots, \Gamma_k$ be the Jordan curves in the $\partial \Omega$ and $p_1, \dots, p_\ell$ be the points.
We may assume that $p_1, \dots, p_\ell \in E$ and set $w_j = f(p_j)$ and $E' = \Gamma_1 \cup \dots \cup \Gamma_k$.
For each point $p_j$ there exists $\widehat{H}_j$, by Theorem~\ref{thm:1.1}, such that $\widehat{H}_j|_{E'} = 0$.
By factoring out potential roots at $p_j$ and scaling we may also assume that $\widehat{H}_j(p_j) = 1$.
Setting
\begin{equation*}
    H_j(z)
    =
    \frac
    {
        (z - p_1)\cdots(z - p_{j - 1})(z - p_{j + 1})\cdots(z - p_\ell)
    }
    {
        (p_j - p_1)\cdots(p_j - p_{j - 1})(p_j - p_{j + 1})\cdots(p_j - p_\ell)
    }
    \widehat{H}_j(z),
    \quad
    z \in \overline{\Omega}
\end{equation*}
we have that $H_j$ is continuous on $\overline{\Omega}$,
    holomorphic on $\Omega$,
    $H_j(p_j) = 1$, and
    $H_j = 0$ on $E \setminus \{p_j\}$.
With $\widehat{F}$ such that $\widehat{F}|_{E'} = f$ we set
\begin{equation*}
    F(z)
    =
    \widehat{F}(z)
    +
    \sum_{j = 1}^\ell
        (f(p_j) - \widehat{F}(p_j))H_j(z),
    \quad
    z \in \overline{\Omega}.
    \qedhere
\end{equation*}
\end{proof}

\bibliographystyle{siam}
\bibliography{rs_bibref}

\end{document}